\documentclass[12pt]{amsart}
\usepackage{url} 
\usepackage{graphicx,enumerate}
\usepackage[latin1]{inputenc}
\usepackage{tikz}
\usetikzlibrary{positioning}
\usetikzlibrary{shapes,arrows}
\theoremstyle{plain} 
\newtheorem{theorem} {Theorem} [section]

\newtheorem{proposition}[theorem]{Proposition}

\theoremstyle{definition}

\theoremstyle{remark}
\newtheorem{remark} {Remark}

\numberwithin{equation}{section}

\def \R{\mathbb R}
\def \Rgez{\R_{\ge 0}}
\def \Rgtz{\R_{> 0}}
\def \Rzero{\mathcal R_0}
\def \Rone{\mathcal R_1}
\def \Rtwo{\mathcal R_2}
\def \Ione{\mathcal I_1}
\def \Itwo{\mathcal I_2}
\def \Ioned{\mathcal I_1^\delta}
\def \Itwod{\mathcal I_2^\delta}
\def \deltaone{\delta^{*}}
\def \deltatwo{\delta^{**}}
\def \aand{\qquad \textrm{and} \qquad}
\def \ccc{\widetilde c}

\begin{document}


\title[A Bare-Bones Mathematical Model of Radicalization]{A Bare-Bones Mathematical Model of Radicalization}


\author[]{C. Connell McCluskey, Manuele Santoprete\\
}



\begin{abstract}
Radicalization is the process by which people come to adopt increasingly extreme political or religious
ideologies.  While radical thinking is by no means problematic in itself, it becomes a threat to national
security when it leads to violence.  We introduce a simple compartmental model (similar to epidemiology
models) to describe the radicalization process.  We then extend the model to allow for multiple ideologies.
Our approach is similar to the one used in the study of multi-strain diseases.  Based on our models, we
assess several strategies to counter violent extremism.
\end{abstract}

\maketitle

\begin{center} \today \end{center}

\section{Introduction}
 Radicalization is the process by which people come to adopt increasingly extreme political or religious ideologies.  
In recent years, radicalization has become a major concern for national security because it can lead to violent extremism.
Governments and security services are under intense pressure to detect and stop individuals undergoing radicalization
before they became violent extremists.  As a consequence, they are making a substantial effort to better understand the
radicalization process and to identify the psychological, social, economic, and political circumstances that lead to violent
extremism.  It is in this context that this paper attempts to describe radicalization mathematically by modeling the spread
of extremist ideology in a manner similar to how the spread of infectious diseases has been modelled.

Vector-borne diseases have been used as a metaphor for Islamist militancy and radicalization by Yacoubian and Stares
\cite{yacoubian2005rethinking}, and by Post \cite{post2006countering}.
In this analogy the agent is the ideology, the vector the recruiters, and the hosts are the members of the vulnerable
population.  This analogy, although forceful, is inexact because transmission in vector-borne diseases (e.g.  malaria)
involves at least two species: the vector and the human host, in which case it is necessary to consider the dynamics of
both species.

A more realistic analogy may be to consider, for instance, tuberculosis as a metaphor for radicalization.  
In tuberculosis we have an agent (Mycobacterium tuberculosis) and a population consisting of susceptible individuals
(who are vulnerable to the agent), latently infected individuals (who have been infected with M. tuberculosis but are
not yet infectious) and infectious individuals (who can transmit the disease).  In this analogy the agent is the ideology,
susceptible individuals are those who have not embraced the extremist ideology, latent individuals are those who have
accepted the extremist ideology to some degree, but are not actively spreading it, and infectious individuals are recruiters
who are spreading the extremist ideology to susceptibles.
This point of view is useful for describing the radicalization process of individuals and the recruitment in terrorist organizations
from a mathematical perspective.

In Section 2 we introduce a compartmental model to describe the radicalization process.  We divide the
population into three compartments, $ (S) $ susceptible, $ (E) $ extremists, and $ (R) $ recruiters.  
This model resembles tuberculosis model studied in \cite{mccluskey2006lyapunov}.

A similar approach was taken in \cite{castillo2003models} where the authors studied the spread of fanatical behaviours as a
type of epidemiological contact process.  Our model, however, is different because it emphasizes the role of recruiters in the
dynamics of terrorist organizations.  
We note that our model, with minor modifications, may be applied to political parties (see \cite{jeffs2016activist} for a
mathematical model of political party growth).  We view this model as bare-bones because we concentrate our attention
only on one aspect of this complex problem, namely, the recruitment process.  

The importance of recruiters in the radicalization process was recognized by President Barack Obama, who in a 2015
op-ed for the LA times \cite{obama_2015} wrote ``We know that military force alone cannot solve this problem. Nor
can we simply take out terrorists who kill innocent civilians. We also have to confront the violent extremists - the
propagandists, recruiters and enablers - who may not directly engage in terrorist acts themselves, but who radicalize,
recruit and incite others to do so."

Using our model, we want to test some possible strategies for countering violent extremism.  A useful parameter that quantifies
the transmission potential of an extremist ideology in our model is the basic reproduction number $\mathcal{R}_0$.  In
mathematical epidemiology it is defined as the number of secondary infections produced when one infectious individual
is introduced into a population consisting only of susceptible individuals.  In Section 2, we will show that when
$\mathcal{R}_0<1 $ the ideology will be eradicated; that is, eventually the number of recruiters and extremist will go to
zero.  When $\mathcal{R}_0 > 1 $ the ideology will become endemic; that is, the recruiters and extremists, will establish
themselves in the population.

In our bare-bones model the basic reproduction number is 
\[
\mathcal{R}_0 = \frac{\Lambda} {\mu} \frac{\beta \left[ c_E + q_R (\mu + d_E) \right]}
		{(\mu + d_E + c_E) (\mu + d_R + c_R) - c_R c_E}, 
\]
where $\mu$ is the mortality rate of the susceptible population, $ d_E $ and $ d_R $ are the additional mortality rates
of the extremists and recruiters, respectively, and $\beta$ is the mass action transmission coefficient of the ideology.  
(Other parameters are described in Section 2.)

One approach to dealing with extremism is to attempt to kill or remove extremists and recruiters, corresponding to
increasing the parameters $d_E$ and $d_R$.  Since $\mathcal{R}_0$ is a decreasing function of these two parameters,
this would in turn decrease $\mathcal{R}_0$.  Consequently, according to our model, these are potentially successful
strategies to counter violent extremism.  A different strategy consists of decreasing the parameter $\beta$ by bolstering
the resilience of communities to recruitment into radicalization and violent extremism.  This is, one aspect of the
``Countering Violent Extremisms" program of the U.S.  Department of Homeland Security \cite{Department2015Countering},
one aspect of the ``Prevent Strategy" \cite{Home2011Prevent} of the British government, and the ``Prevent" element of the
``Building Resilience Against Terrorism" strategy of the Government of Canada \cite{safety2013building}.

Another strategy we want to test consists of supporting an alternative, more moderate ideology that competes with the
terrorist ideology.  This is, for instance, another aspect of the ``Prevent" element of the ``Building Resilience Against Terrorism"
strategy of the Government of Canada \cite{safety2013building}.  In fact, one of the desired outcomes listed for the ``Prevent"
element \cite{safety2013building} is the following: ``Violent extremist ideology is effectively challenged by producing effective
narratives to counter it."  This approach is also advocated by Yacoubian and Stares \cite{yacoubian2005rethinking} and by
Post \cite{post2006countering}.

To study the effectiveness of this approach, in Section 3, we generalize our model to include two ideologies under
the assumption that individuals cannot move from one ideology to the other.  Based on how parameters are chosen,
we may have one moderate ideology and one extreme, or two extreme ideologies.  For this model there are two basic
reproduction numbers $\Rone$, associated with the first ideology, and $\Rtwo$ associated
with the second ideology.  We will prove that if both reproduction number are less than one then both ideologies
will eventually be eradicated.  If, however, one of the reproduction numbers is larger than the other, say
$\Rone > \Rtwo$, then the second ideology will be eliminated; that is, the number of recruiters and extremist
supporting the second ideology will eventually go to zero.  This can happen even if $\Rtwo>1 $.  This is a version
of the so called principle of competitive exclusion of ecology and epidemiology \cite{martcheva2015introduction}.
Our analysis suggests that introducing and supporting another ideology, in conjunction with the approaches we
mentioned above (i.e., targeting terrorists and recruiters and increasing the resilience of communities to radical
ideologies), may bolster the overall effectiveness of strategies to counter violent extremism.

Finally, in Section 4, we amend the two ideology model to take into account that individuals subscribing to one ideology
may change their mind and commit to another ideology, this is done by introducing a cross-interaction term.  We assume
that the flow of individuals is from ideology one to ideology two.  If we take the second ideology to be the moderate
one, we can use this cross interaction term to model de-radicalization.  By de-radicalization we mean, following Horgan
\cite{horgan2009walking}, the social and psychological process whereby an individual's commitment to, and involvement
in, violent radicalization is reduced to the extent that they are no longer at risk of involvement and engagement in violent
activity.  There are many factors that may contribute to de-radicalization, including, possibly, some of the programs developed
by the British, Canadian and U.S.  Governments \cite{Home2011Prevent,safety2013building,Department2015Countering}.
If, instead, we take the first ideology to be the moderate one, we can use the cross-interaction term to model individuals
subscribing to a moderate ideology that change their mind and become radicalized.



Naturally, we are interested in understanding how the presence of the cross-interaction term changes the results obtained
in Section 3.  There are two main changes.  First of all, it is possible to have a coexistence equilibrium, where both ideologies
persist.  This, of course, may be an unwelcome modification since our goal is to eradicate the more extreme of the ideologies.
If the coexistence equilibrium is locally asymptotically stable this makes eliminating either ideology difficult.  Second, recall
that in Section 3 the ideology with the larger reproduction number dominates, and the other ideology eventually vanishes.
With the cross-interaction term, even if $\Rone > \Rtwo$, it is possible that the second ideology can come to dominate.
This is because the cross-interaction term gives a competitive advantage to the second ideology.  The main lesson we take
away from this is that, based on our analysis of Section 3, introducing a second more moderate ideology may considerably
enhance other efforts to counter violent extremism.  If cross-interaction terms are taken into account, however, based on our
analysis of Section 4, such advantages may be reduced if the moderate ideology is the first one, and it becomes a pathway
to the second (more extreme) ideology.

\section{Bare-Bones Model}
\subsection{Equations}
We model the spread of extreme ideology as a contact process.  We use a compartmental model (similar to
that which is used to model epidemics) to describe the dynamics.  We assume that within the full population
there is a subpopulation that is potentially at risk of adopting the ideology.  (For example, for the ideology of
violent resistance to British rule in Northern Ireland during the late 20th century for the purpose of joining the
Republic of Ireland, the subpopulation that was potentially at risk of adopting the ideology would primarily
consist of the Catholic subpopulation of Northern Ireland.)
We further divide this subpopulation into three compartments:
\begin{itemize}
 \item susceptibles - $S$
 \item extremists - $E$
 \item recruiters - $R$.
\end{itemize}
The transfer diagram for the system appears in Figure \ref{fig:transfer_diagram}.

The susceptible class consists of individuals that have not adopted the extremist ideology.  The extremist group
includes violent adopters of the ideology that engage in terrorists acts.  The recruiters group includes those that
Barack Obama, former President of the United States of America defined as ``---the propagandists, recruiters
and enablers--- who may not directly engage in terrorist acts themselves, but who radicalize, recruit and incite
others to do so"  \cite{obama_2015}.


\tikzstyle{block} = [rectangle, draw, fill=blue!20, 
 text width=5em, text centered, rounded corners, minimum height=4em]
\tikzstyle{line} = [draw, -latex]
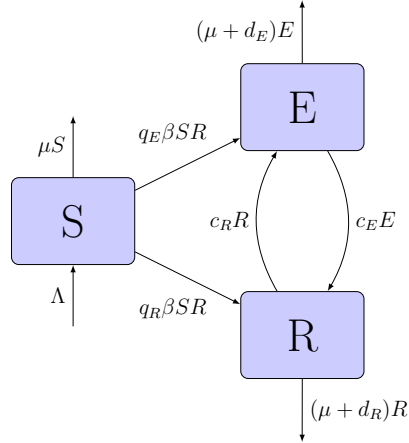
\begin{figure}[h] 
 \scalebox{0.7}{
\begin{tikzpicture}[node distance = 2cm,auto]
 \node [block] (S) {\Huge{S}}; 
 \node [block, below right=0.5cm and 2cm of S] (R) {\Huge{R}};
 \node [block, above right=0.5cm and 2cm of S] (E) {\Huge{E}};
 
 \path [line] (S) --node [below left, near end ]{$ q_R \beta SR $} (R);
 \path [line] (S) -- node [above left,near end ]{$q_E \beta S R $} (E); 
 \path [line] (S) -- node{$ \mu S $} (0,2);
 
 \path [line] (R) edge [bend left] node [left]{$c_R R $} (E);
 \path [line] (E) edge [bend left] node [right]{$c_E E $} (R);
 \path [line] (0,-2) -- node{$ \Lambda $} (S);

 \path [line] (R) -- node [] {$(\mu + d_R) R $} (4.35,-4.2);
 \path [line] (E) -- node [] {$(\mu + d_E) E $} (4.35,4.2);
\end{tikzpicture}
}
\caption{Transfer diagram for the bare-bones model.  It shows the progression of radicalization from the Susceptible
$(S)$ class to the extremist $(E)$, and recruiter $(R)$ classes.  \label{fig:transfer_diagram}}
\end{figure}

New individuals enter the population into the susceptible group at the constant rate $\Lambda$.
%
The per capita death rate for natural causes is $\mu$, so that the average life expectancy is $1/ \mu$.
Extremists and recruiters have additional removal rates, due to being captured or killed, with rate constants
$ d_E $ and $ d_R $, respectively.


We assume that susceptibles and recruiters interact according to a mass action law, and that the rate at
which susceptibles are recruited to adopt the extremist ideology is proportional to the number of interations
that are occuring.  Thus, susceptibles are recruited at rate $\beta S R$, with a fraction $q_E$
entering the extremist class and a fraction $q_R = 1 - q_E$ entering the recruiter class.

Extremists switch to the recruiter class with rate constant $ c_E $, while recruiters enter the extremist
class with rate constant $ c_R $.  
Thus, the radicalization model consists of the following differential equations together with
non-negative initial conditions:
\begin{equation}\label{eqn:model}
 \begin{aligned}
 S' & = \Lambda - \mu S - \beta SR		\\
 E' & = q_E\beta SR - (\mu + d_E + c_E) E + c_R R		\\
 R' & = q_R \beta SR + c_E E - (\mu + d_R + c_R) R, 
 \end{aligned}
\end{equation}
where $ q_E + q_R = 1 $, $ q_E, q_R \in [0,1] $ and all other parameters are positive.

Note that although the problem we are addressing is similar to the one studied in \cite{castillo2003models},
the model considered in \cite{castillo2003models} does not include the recruiters.

\begin{proposition}\label{prop:invariant_region} 
 Under the flow described by \eqref{eqn:model}, the region
 \[
 \Delta = \left \{ (S, E, R) \in \mathbb{R}^3_{\geq 0} : S+E+I \leq \frac{\Lambda} {\mu} \right \}
 \]
 is positively invariant and attracting within $\mathbb{R}_{\ge 0}^3$
 (i.e.  all solutions starting in $ \Delta $ remain in $ \Delta $ for all $ t > 0 $, while solutions starting outside $\Delta$ either
 enter or approach in the limit).
\end{proposition}

\begin{proof}
 We check the direction of the vector field along the boundary of $ \Delta $.  Along $ S = 0 $ we have 
 $ S' = \Lambda >0$ so the vector field points inwards.  Along $ E = 0 $ we have $ E' = q_E \beta S R + c_R R \geq 0 $
 provided $ R, S \geq 0 $.  Similarly, along $ R = 0 $ we have $ R' = c_E E \geq 0 $ provided $ E \geq 0 $.
 It follows from \cite[Proposition 2.1]{Haddad+2010} that $\mathbb{R}_{\ge 0}^3$ is positively invariant.
 
 Let $T = S + E + R$.  Then
 \[
 T' = S' + E' + R' = \Lambda - \mu T - d_R R - d_E E\leq \Lambda - \mu T.
\]
Using a standard comparison theorem, it follows that 
\begin{equation}				\label{solution for T}
 T (t) \leq \left( T (0) - \frac{\Lambda}{\mu} \right) e^{ - \mu t} + \frac{\Lambda} {\mu}.
\end{equation}
for $t \ge 0$.
Thus, if $ T (0) \leq \frac{\Lambda} {\mu} $, then $ T (t) \leq \frac{\Lambda} {\mu} $ for all $t \ge 0$.
Hence, the set $ \Delta $ is positively invariant.
Furthermore, it follows from \eqref{solution for T}, that $\limsup_{ t \to \infty} T \leq \frac{\Lambda} {\mu}$,
demonstrating that $\Delta$ is attracting within $\mathbb{R}_{\ge 0}^3$.
\end{proof}

\subsection{Equilibrium Points}\label{ssec:eq_points}

There are no equilibria with $S$ equal to zero.  The only equilibrium with either $E$ or $R$ equal to
zero is $x_0 = \left( S_0 ,0, 0 \right) = \left( \frac{\Lambda} {\mu} ,0, 0 \right)$.  To find other equilibria
$x^* = (S^*, E^*, R^*) \in \R_{>0}^3$, we first set $E ' = R ' = 0$, treating $S^*$ as a parameter.
This gives the linear system 
\begin{equation}\label{eqn:matrix} 
 \begin{aligned}
 \begin{bmatrix}
 - (\mu + d_E + c_E) &  q_E\beta S^* + c_R \\
 c_E & q_R\beta S^* - (\mu + d_R + c_R )
 \end{bmatrix} \begin{bmatrix}
 E^* \\
 R^* 
\end{bmatrix} = \begin{bmatrix}
 0 \\
 0 
 \end{bmatrix} 
 \end{aligned} 
 \end{equation} 
In order to have non-zero solutions for $E^*$ and $R^*$, the coefficient matrix must be singular and have
determinant zero.  Recalling that $q_E + q_R = 1$, this gives 
\begin{equation}		\label{S^*}
 S^* = \frac{ ( \mu + d_E + c_E) (\mu + d_R + c_R) - c_E c_R} {\beta[ c_E + q_R (\mu + d_E)]}
 = \frac{1} {\Gamma},
\end{equation}
where
\[
\Gamma = \frac{\beta [ c_E + q_R (\mu + d_E)]}{ ( \mu + d_E + c_E) (\mu + d_R + c_R) - c_E c_R}.
\]
With this value of $S^*$, there is a infinite number of solutions to \eqref{eqn:matrix}.  To find a positive
equilibrium, we take the first line of \eqref{eqn:model} and solve for $R^*$ obtaining
\[
 R^* = \frac{\Lambda - \mu S^*} {\beta S^*} = \frac{\mu} {\beta} \left( \frac{\Lambda \Gamma} {\mu} - 1 \right),
\]
which is positive for $\mathcal{R}_0 = \frac{\Lambda \Gamma} {\mu} >1 $, where $\mathcal{R}_0$
is the basic reproduction number that we will calculate in the next section.

Adding $E'$ and $R'$ at an equilibrium gives 
\[
0 = \beta S^* R^* - (\mu + d_E) E^* - (\mu + d_R) R^*
\]
Isolating $E^*$, and noting that $\beta S^* = \frac{\beta}{\Gamma}$, we find
\begin{equation}				\label{E^*}
 E^* = \frac{\frac{\beta} {\Gamma} - (\mu + d_R)} {\mu + d_E} R^*.
\end{equation}
Note that
\begin{equation*}
\Gamma
< \frac{\beta [ \mu + d_E + c_E ]} { (\mu + d_E + c_E) (\mu + d_R )}
 = \frac{\beta} {\mu + d_R}.  
\end{equation*} 
Combining this observation with \eqref{E^*}, we see that $E^*$ and $R^*$ have the same sign (which depends
on the sign of $\Rzero-1$).

\begin{theorem}
If $\Rzero < 1$, then the only equilibrium is $x_0 = \left( \frac{\Lambda} {\mu} ,0, 0 \right)$.
If $\Rzero > 1$, then there are two equilibria: $x_0$ and $x^* = (S^*, E^*, R^*) \in \R_{>0}^3$,
\end{theorem}

\subsection{The basic reproduction number $\mathcal{R}_0$}
The basic reproduction number $\mathcal{R}_0$ is the spectral radius of the next generation matrix $N$,
and can be calculated as follows (see \cite{van2002reproduction} for a full description of the method).
The next generation matrix is given by $ N = F V^{ - 1}$ with 
\[
 F = \begin{bmatrix}
 \frac{\partial \mathcal{F}_E} {\partial E} & \frac{\partial \mathcal{F}_E} {\partial R} \\[1em]
 \frac{\partial \mathcal{F}_R} {\partial E} & \frac{\partial \mathcal{F}_R} {\partial R} 
 \end{bmatrix} (x_0)
 \aand
 V = \begin{bmatrix}
 \frac{\partial \mathcal{V}_E} {\partial E} & \frac{\partial \mathcal{V}_E} {\partial R} \\[1em]
 \frac{\partial \mathcal{V}_R} {\partial E} & \frac{\partial \mathcal{V}_R} {\partial R} 
 \end{bmatrix} (x_0).
\]
Here, $\mathcal{F}_E $ and $\mathcal{F}_R $ are the rates of appearance of newly radicalized individuals
in the classes $ E $ and $ R $, respectively.  The movement between $E$ and $R$ is given by $\mathcal{V}_E $
and $\mathcal{V}_R $.  In our case 
\[
 \mathcal{F} = \begin{bmatrix}
 \mathcal{F}_E \\
 \mathcal{F}_R 
 \end{bmatrix} = 
 \begin{bmatrix}
 q_E \beta SR \\
 q_R \beta SR 
 \end{bmatrix} 
\aand
\mathcal{V} = \begin{bmatrix}
 \mathcal{V}_E \\
 \mathcal{V}_R 
 \end{bmatrix} = 
 \begin{bmatrix}
 (\mu + d_E + c_E) E - c_R R \\
  - c_E E + (\mu + d_R + c_R) R
 \end{bmatrix},
\]
so that $E' = \mathcal{F}_E - \mathcal{V}_E$ and $R' = \mathcal{F}_R - \mathcal{V}_R$.
Hence
\[
 F =
 \beta S_0 \begin{bmatrix}
 0& q_E \\
 0 & q_R 
 \end{bmatrix}
\aand
 V = \begin{bmatrix}
 (\mu +d_E + c_E) & - c_R \\
 - c_E & (\mu + d_R + c_R)
 \end{bmatrix}.
\]
We now have
\[
 N = F V^{-1} = \frac{\beta S_0}{(\mu + d_E + c_E) (\mu + d_R + c_R) - c_R c_E} 
 \begin{bmatrix}
 0 & q_E \\
 0& q_R
 \end{bmatrix}
 \begin{bmatrix}
 (\mu +d_R + c_R) & c_R \\
 c_E & (\mu + d_E + c_E) 
 \end{bmatrix} 
\]
Note that $F$ has rank $1$ and so the same is true of $N$.  Since one eigenvalue of $N$ is
zero the spectral radius is equal to the absolute value of the remaining eigenvalue.  Since the
trace is equal to the sum of the eigenvalues and there is only one non-zero eigenvalue, we
see that the spectral radius of $N$ is equal to the absolute value of the trace (which happens
to be positive).  Thus,
\begin{equation}				\label{Rzero}
\mathcal{R}_0 = S_0 \frac{\beta [ c_E + q_R (\mu + d_E)]}{ ( \mu + d_E + c_E) (\mu + d_R + c_R) - c_E c_R}
 = \frac{\Lambda} {\mu} \Gamma.
\end{equation}
The following result follows immediately from \cite[Theorem 2]{van2002reproduction}.

\begin{theorem}
If $\Rzero < 1$, then $x_0$ is locally asymptotically stable.  If $\Rzero > 1$, then $x_0$ is unstable.
\end{theorem}

\subsection{Global stability of $x_0$ for $\Rzero \le 1$}
For the calculations of this section and later sections, it is convenient to define
\[
\ccc = c_E + q_R (\mu + d_E)
\aand
D = (\mu + d_R + c_R) (\mu + d_E + c_E) - c_E c_R.
\]
This allows us, for instance, to write $\Rzero = \frac{\beta \Lambda \ccc}{\mu D}$.
We will now study the global stability of $x_0$ by taking the Lyapunov function 
\[
U = S_0 g \left( \frac{S}{S_0} \right) + \frac{c_E}{\ccc} E + \frac{\mu + d_E + c_E}{\ccc} R,
\]
where
\begin{equation}		\label{g}
g (x) = x - 1 - \ln (x).
\end{equation}
Note that $\left. S' \right|_{S=0} = \Lambda > 0$.  Combining this with the fact that $\Delta$ is compact,
positively invariant and attracting, it follows that $\liminf S(t) > 0$, and so we may assume that $S>0$.
Thus, the singularity in $g$ at $0$ is not a problem.

Differentiating $U$, and recalling that $\Lambda = \mu S_0$, we find
\begin{align*}
 U' & = \left( 1 - \frac{S_0}{S} \right) S' + \frac{c_E}{\ccc} E' + \frac{\mu + d_E + c_E}{\ccc} R'		\\
 & = \left( 1 - \frac{S_0}{S} \right) \left( \Lambda - \mu S - \beta S R \right)
 		+ \frac{c_E}{\ccc} \left( q_E \beta S R - \left( \mu + d_E + c_E \right) E + c_R R \right)	\\
 & \hspace{4.0cm} + \frac{\mu + d_E + c_E}{\ccc} \left( q_R \beta S R + c_E E - \left( \mu + d_R + c_R \right) R \right)	\\
 & = \left( 1 - \frac{S_0}{S} \right) \left( \mu S_0 - \mu S\right)
 		+ \beta S_0 R + \beta S R \left[ - 1 + \frac{q_E c_E + q_R (\mu + d_E + c_E)}{\ccc} \right]		\\
 & \hspace{4.0cm} + \frac{c_E c_R - (\mu + d_E + c_E) \left( \mu + d_R + c_R \right)}{\ccc} R	\\
 & = - \mu \frac{\left( S - S_0\right)^2}{S} + \left[ \beta S_0 - \frac{D}{\ccc} \right] R	\\
 & = - \mu \frac{\left( S - S_0\right)^2}{S} + \frac{D}{\ccc} \left[ \frac{\beta \Lambda \ccc}{\mu D} - 1 \right] R	 \\
 & = - \mu \frac{\left( S - S_0\right)^2}{S} + \frac{D}{\ccc} \left[ \Rzero - 1 \right] R	 \\
\end{align*}
For $\Rzero \le 1$, we have $U ' \le 0 $ with equality only if $S = S_0$.  However, the largest invariant set
for which $S$ remains constant at $S_0$ consists of just the equilibrium $x_0$.  Thus, by LaSalle's Invariance
Principle, we obtain the following theorem.

\begin{theorem}\label{thm:stable0} 
 If $ \Rzero \le 1 $, then $x_0$ is globally asymptotically stable on $\mathbb{R}^3_{\geq 0}$.  
\end{theorem}

\subsection{Global stability of $x^*$ for $\Rzero > 1$}

\begin{theorem}\label{thm:xast_stability}
 If $\Rzero > 1$, then the equilibrium $ x^* $ is globally asymptotically stable on $\mathbb{R}^3_{> 0} $.  
\end{theorem} 

\begin{proof} We study the global stability of $ x^* $ by considering the Lyapunov function

 \[
 W = S^* g \left( \frac{ S} { S^*} \right)+AE^* g \left( \frac{ E} { E^*} \right)
 		+BR^* g \left( \frac{ R} { R^*} \right)
\]
where $g$ is given in \eqref{g}, and the constants $A$ and $B$ are given below.  Differentiating $W$ along
solutions of \eqref{eqn:model} gives 
\begin{align*} 
 W' = & \left( 1 - \frac{ S^*} { S} \right) S' + A \left( 1 - \frac{ E^*} { E} \right) E' + B \left( 1 - \frac{ R^*} { R} \right) R'\\
 = & \left( 1 - \frac{ S^*} { S} \right) \left[ \Lambda - \mu S - \beta SR	\right]+ A \left( 1 - \frac{ E^*} { E} \right) \left[ q_E\beta SR - (\mu + d_E + c_E) E + c_R R	\right] \\
 &\hspace{1.0cm} + B \left( 1 - \frac{ R^*} { R} \right) \left[ q_R \beta SR + c_E E - (\mu + d_R + c_R) R \right].
\end{align*} 
Note that by considering Equation \eqref{eqn:model} at $x^*$ we have
\begin{align*}
 \Lambda & = \mu S^* + \beta S^* R^*					\\ 
 (\mu + d_E + c_E) E^* & = q_E \beta S^* R^* + c_R R^*		\\
 (\mu + d_R + c_R)R^* & = q_R \beta S^* R^* + c_E E^*
\end{align*}
Using these relationships to rewrite $W'$ yields 
\begin{equation*} 
 \begin{aligned}
 W' = & \left( 1 - \frac{ S^*} { S} \right)\left[ \mu S^* \left( 1 - \frac{ S} { S^*} \right)
 				+ \beta S^* R^* \left( 1 - \frac{ S R} { S^* R^*}\right) \right]\\
 & + A \left( 1 - \frac{ E^*}{E} \right) \left[ q_E \beta S^* R^* \left( \frac{ S R} { S^* R^*} - \frac{ E} { E^*}\right)
 				+ c_{ R} R^* \left( \frac{ R} { R^*} - \frac{ E} { E^*} \right) \right]\\
 & + B \left( 1 - \frac{ R^*} { R} \right) \left[q_R \beta S^* R^* \left( \frac{ SR} { S^* R^*} - \frac{ R} { R^*} \right)
 				+ c_E E^* \left( \frac{ E} { E^*} - \frac{ R} { R^*} \right) \right]
 \end{aligned} 
\end{equation*} 
To make the equations more concise, we define $s = \frac{S}{S^*}$, $e = \frac{E}{E^*}$ and $r = \frac{R}{R^*}$.
Then
\begin{equation*}
 \begin{aligned}
 W' = & \mu S^* \left[ 2 - s - \frac{1}{s} \right] + A c_R R^* \left[ r - e - \frac{r} {e} + 1 \right]
 											+ B c_E E^* \left[ e - r - \frac{e} {r} + 1 \right]	\\
 	& + \beta S^* R^* \left[ \left( 1 - \frac{1} {s} - s r + r \right) + A q_E \left( s r - e - \frac{s r} {e} + 1 \right)
											+ B q_R \left( s r - r - s + 1 \right) \right].
 \end{aligned} 
\end{equation*}
We now choose $A, B > 0$ so that $ A q_E + B q_R = 1 $, and $ B c_E E^* = A c_R R^* + A q_E \beta S^* R^* $.
The first of these conditions allows us to distribute the term $\left( 1 - \frac{1} {s} - s r + r \right)$ into the two terms
with coefficients $A q_E$ and $B q_R$.  The second condition allows us to replace $B c_E E^*$.  We now have
\begin{equation*}
\begin{aligned}
 W' = & \mu S^* \left[ 2 - s - \frac{1}{s} \right] + A c_R R^* \left[ 2 - \frac{r} {e} - \frac{e} {r} \right]	\\
 	& + \beta S^* R^* \left[ A q_E \left( 3 - \frac{1} {s} - \frac{e} {r} - \frac{s r} {e} \right)
					+ B q_R \left( 2 - s - \frac{1} {s} \right) \right],
\end{aligned} 
\end{equation*}
which is less or equal to zero by the arithmetic mean--geometric mean inequality, with equality if and only if
$S = S^*$, $E/R = E^*/ R^*$.  The largest invariant set for which $S$ remains constant at $S^*$ consists of
just the equilibrium $x^*$.  Thus, by LaSalle's Invariance Principle, we may conclude that $ x^* $ is globally
asymptotically stable.
\end{proof}

\section{Multiple Ideology Model}	\label{Section 3}
A situation that may preclude the spread of a violent extremist ideology is the presence of a competing more moderate
ideology.  In this section we want to study whether or not this approach may work.  In ecology a similar occurrence
is often governed by the {\it competitive exclusion principle}.  This principle states that {\it competitors cannot
coexist in the same niche}.  A similar principle holds in some epidemiological systems \cite{martcheva2015introduction}.

\subsection{Equations}
We consider a population that is influenced by two ideologies.  The population of interest is now divided into
five compartments:
\begin{itemize}
 \item susceptibles - $S$
 \item extremists following the first ideology - $E_1$
 \item recruiters for the first ideology - $R_1$.
 \item extremists following the second ideology - $E_2$
 \item recruiters for the second ideology - $R_2$.
\end{itemize}
The transfer diagram for the system appears in Figure \ref{fig:transfer_diagram1}.  The differential equations
that govern the system are:
\begin{equation}			\label{eqn:model1}
 \begin{aligned}
 S' & = \Lambda - \mu S - \beta_1 SR_1	- \beta_2 SR_2	\\
 E_1' & = q_{E_1} \beta_1 SR_1 - (\mu + d_{E_1} + c_{E_1}) E_1 + c_{R_1} R_1 		\\
 R_1' & = q_{R_1} \beta_1 S R_1 + c_{E_1} E_1 - (\mu + d_{R_1} + c_{R_1}) R_1\\ 
 E_2' & = q_{E_2} \beta_2 SR_2 - (\mu + d_{E_2} + c_{E_2}) E_2 + c_{R_2} R_2	\\
 R_2' & = q_{R_2} \beta_2 S R_2 + c_{E_2} E_2 - (\mu + d_{R_2} + c_{R_2} )R_2, 
 \end{aligned}
\end{equation}
where $ q_{E_1} + q_{R_1} = 1 $, $ q_{E_1}, q_{R_1} \in [0,1] $, $ q_{E_2} + q_{R_2} = 1 $, $ q_{E_2}, q_{R_2} \in [0,1]$, and all other parameters
are positive.
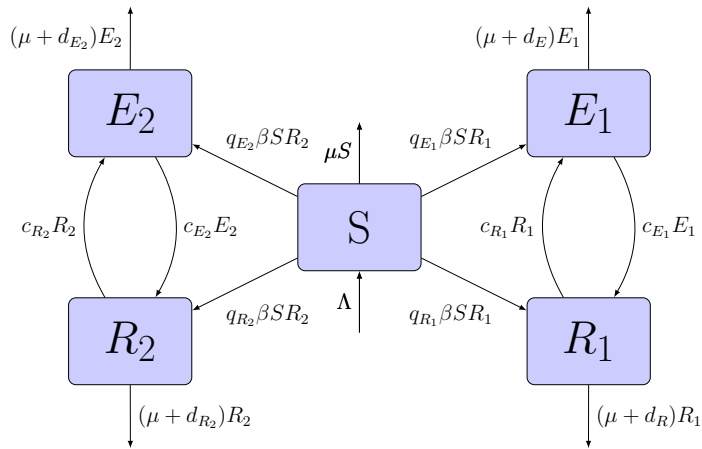
\begin{figure}[h] 
 \scalebox{0.7}{
\begin{tikzpicture}[node distance = 2cm,auto]
 \node [block] (S) {\Huge{S}}; 
 \node [block, below right=0.5cm and 2cm of S] (R1) {\Huge{$ R_1 $}};
 \node [block, above right=0.5cm and 2cm of S] (E1) {\Huge{$ E_1 $}};
 \node [block, below left=0.5cm and 2cm of S] (R2) {\Huge{$R_2$}};
 \node [block, above left=0.5cm and 2cm of S] (E2) {\Huge{$E_2$}};

 
 \path [line] (S) --node [below left, near end ]{$ q_{R_1} \beta SR_1 $} (R1);
 \path [line] (S) -- node [above left,near end ]{$q_{E_1} \beta S R_1 $} (E1); 
 \path [line] (S) -- node{$ \mu S $} (0,2);
 
 \path [line] (R1) edge [bend left] node [left]{$c_{R_1} R_1 $} (E1);
 \path [line] (E1) edge [bend left] node [right]{$c_{E_1} E_1 $} (R1);
 \path [line] (0,-2) -- node{$ \Lambda $} (S);
 
 \path [line] (R1) -- node [] {$(\mu + d_R) R_1 $} (4.35,-4.2);
 \path [line] (E1) -- node [] {$(\mu + d_E) E_1 $} (4.35,4.2);


 \path [line] (S) --node [below right, near end ]{$ q_{R_2} \beta SR_2 $} (R2);
 \path [line] (S) -- node [above right,near end ]{$q_{E_2} \beta S R_2 $} (E2); 
 \path [line] (S) -- node{$ \mu S $} (0,2);
 
 \path [line] (R2) edge [bend left] node [left]{$c_{R_2} R_2 $} (E2);
 \path [line] (E2) edge [bend left] node [right]{$c_{E_2} E_2 $} (R2);
 \path [line] (0,-2) -- node{$ \Lambda $} (S);
 
 \path [line] (R2) -- node [] {$(\mu + d_{R_2} ) R_2 $} (-4.35,-4.2);
 \path [line] (E2) -- node [] {$(\mu + d_{E_2} ) E_2 $} (-4.35,4.2);

\end{tikzpicture}
}
\caption{Transfer diagram for the model with two competing ideologies.  \label{fig:transfer_diagram1}}
\end{figure}

Let $ T = S + R_1 + E_1 + R_2 + E_2$ be the total population.  The proof of the following result is
similar to the proof of Proposition \ref{prop:invariant_region}.

\begin{proposition}\label{prop:invariant_region1} 
The region
$\widetilde \Delta = \left\{ (S, E_1, R_1, E_2, R_2 ) \in \mathbb{R}^5_{\geq 0} : T \leq \frac{\Lambda}{\mu} \right\}$
is positively invariant and attracting within $\mathbb{R}_{\ge 0}^5$ for the flow of \eqref{eqn:model1}.
\end{proposition}

\subsection{Equilibrium Points and Basic Reproduction Numbers}			\label{sec-3-equil}
Generally, System \eqref{eqn:model1} has up to three equilibria.  The computations are similar to the ones
shown in the previous section.  The first equilibrium point is an ideology-free equilibrium:
\[
 x_0 = (S, E_1, R_1, E_2, R_2) = \left( \frac{\Lambda} {\mu}, 0, 0, 0, 0 \right).
\]

In order to describe the other two equilibria, it is useful to first define the basic reproduction number for each
ideology.  (These quantities will be calculated derived using the next generation matrix method in Subsection
\ref{R1, R2 and more}.)  For $i=1,2$, let
\begin{equation*}
\mathcal{R}_i = \frac{\beta_i \Lambda \ccc_i}{\mu D_i},
\end{equation*}
where $\ccc_i = c_{E_i} + q_{R_i} \left( \mu + d_{E_i} \right)$ and
$D_i =  (\mu + d_{R_i} + c_{R_i}) (\mu + d_{E_i} + c_{E_i}) - c_{E_i} c_{R_i}$.
It is also useful to define
\begin{equation*}
\Gamma_i = \frac{\beta_i \ccc_i}{D_i},
\end{equation*}
for $i = 1, 2$, which leads to $\mathcal{R}_i = \frac{\Lambda}{\mu} \Gamma_i$.

For $\Rone > 1$, there is an equilibrium $x^*$ in which extremists and recruiters for ideology one are present
(i.e. $E_1, R_1 \neq 0$), while ideology two is absent (i.e. $E_2 = R_2 = 0$).  This is called the {\it ideology-one
dominance equilibrium}, and is given by 
\[
 x^* = \left( S^*, E_1^*, R_1^*, 0, 0 \right)
 	= \left(\frac{1} {\Gamma_1}, \frac{ \frac{\beta_1} {\Gamma_1} - (\mu + d_{R_1})}{\mu + d_{E_1}} R^*_1,
		\frac{\mu} {\beta_1} \left(\Rone - 1 \right), 0, 0 \right).
\]

For $\Rtwo > 1$, there is an equilibrium $x^{**}$ in which extremists and recruiters for ideology one are present
(i.e. $E_2, R_2 \neq 0$), while ideology two is absent (i.e. $E_1 = R_1 = 0$).  This is called a {\it ideology-two
dominance equilibrium}, and is given by 
\[
 x^{**} = (S^{**}, 0, 0, E_2^{**}, R_2^{**})
 		= \left( \frac{1} {\Gamma_2}, 0, 0,
			\frac{\frac{\beta_2} {\Gamma_2} - ( \mu + d_{R_2})} {\mu + d_{E_2}} R^{**}_2,
			\frac{\mu} {\beta_2} \left(\Rtwo - 1 \right) \right).
\]
The existence conditions for these equilibria are summarized in the following theorem.

\begin{theorem}
 The ideology-free equilibrium $x_0$ always exists.
 In addition, there is an ideology-one dominance equilibrium $x^*$ if and only if $\Rone > 1$.
 Similarly, there is an ideology-two dominance equilibrium $x^{**}$ if and only if $\Rtwo >1$.
\end{theorem}

An equilibrium in which practitioners of each type are present is called a {\it coexistence equilibrium}.
For a coexistence equilibrium the equations $E_1' = R_1 ' = E_2 ' = R_2 ' = 0$ must be satisfied.
Treating $S$ as a parameter we can view each of $E_1' = R_1 ' = 0$ and $E_2 ' = R_2 ' = 0$ as
homogeneous linear systems in two variables.  Since we are looking for non-trivial solutions of such
systems the same argument used in Subsection \ref{ssec:eq_points} shows that the first system requires 
\[
 S = \frac{1} {\Gamma_1} = \frac{\Lambda} {\mu} \frac{1} {\Rone}, 
\]
while the second system requires 
\[
 S = \frac{1} {\Gamma_2} = \frac{\Lambda} {\mu} \frac{1} {\Rtwo}.  
\]
These two expression for $S$ are consistent if and only if $\Rone = \Rtwo$.  Hence, in the generic case
in which the two reproduction numbers are different there is no coexistence equilibrium.

\subsection{Basic Reproduction Numbers and Invasion Numbers}	\label{R1, R2 and more}
There are two equilibria $x_0$ and $x^{**}$ for which the first ideology is absent.  Thus, it is reasonable to
consider whether the first ideology is able to invade near each of these equilibria.  This is determined by
studying the next generation matrix.  (Similar calculations for the second ideology will follow.)

Using the terminology in \cite{van2002reproduction}, the variables $E_1$ and $R_1$ will be considered as ``infected"
variables for the first ideology.  We will describe them as ``adopters" while still using the approach and notation
described in \cite{van2002reproduction}.  The rates of new adoptions of the first ideology is given by $\mathcal F$
and the movement between the groups of adopters is given by $\mathcal V$, with
\[
 \mathcal{F} = \begin{bmatrix}
 \mathcal{F}_{E_1} \\
 \mathcal{F}_{R_1} 
 \end{bmatrix} = 
 \begin{bmatrix}
 q_{E_1} \beta_1 SR_1 \\
 q_{R_1} \beta_1 SR_1 
 \end{bmatrix} 
\aand
\mathcal{V} = \begin{bmatrix}
 \mathcal{V}_{E_1} \\
 \mathcal{V}_{R_1}
 \end{bmatrix} = 
 \begin{bmatrix}
 (\mu + d_{E_1} + c_{E_1}) E_1 - c_{R_1} R_1 \\
 - c_{E_1} E_1 + (\mu + d_{R_1}+ c_{R_1} ) R_1 
 \end{bmatrix}.
\]
Let $\bar S$ be the value of $S$ at the equilibrium of interest - either $x_0$ or $x^{**}$.  Then
\[
 F = \beta_1 \bar S \begin{bmatrix}
 0& q_{E_1} \\
 0 & q_{R_1}
 \end{bmatrix}
\aand
V = \begin{bmatrix}
 (\mu +d_{E_1} + c_{E_1} ) & - c_{R_1} \\
 - c_{E_1} & (\mu + d_{R_1} + c_{R_1} )
 \end{bmatrix},
\]
The next generation matrix $N$ is given by
\[
 N = F V^{-1}
 =\frac{\beta_1 \bar S }{D_1} 
 \begin{bmatrix}
 0 & q_{E_1} \\
 0& q_{R_1}
 \end{bmatrix}
 \begin{bmatrix}
 (\mu +d_{R_1} + c_{R_1}) & c_{R_1} \\
 c_{E_1} & (\mu + d_{E_1} + c_{E_1}) 
 \end{bmatrix} .
\]
One eigenvalue of $N$ is $0$ and so the other eigenvalue is equal to the trace, which is $\bar S \Gamma_1$.
Thus, the expected number of new adopters of the first ideology, caused by a single adherent to the ideology
near $x_0$, is given by
\[
\Rone = S_0 \Gamma_1 = \frac{\Lambda}{\mu} \Gamma_1,
\]
the same as was defined in the Subsection \ref{sec-3-equil}.  Similarly, the expected number of new adopters of the
first ideology, caused by a single adherent to the ideology, when the system is near $x^{**}$, is given by
\[
\Ione = S^{**} \Gamma_1 = \frac{\Gamma_1}{\Gamma_2}.
\]
$\Ione$ is called the {\it invasion number} for the first ideology.  The sign of $\Ione-1$ determines whether
or not the first ideology can invade near $x^{**}$.  For the second ideology, we can similarly calculate
\[
\Rtwo = \frac{\Lambda}{\mu} \Gamma_2
\aand
\Itwo = \frac{\Gamma_2}{\Gamma_1}.
\]
It now follows that
\[
\Ione = \frac{\Rone}{\Rtwo}
\aand
\Itwo = \frac{\Rtwo}{\Rone}.
\]
By \cite[Theorem 2]{van2002reproduction}, we have the following result.
\begin{theorem}
The following table outlines the existence and local stability of the equilibria $x_0$, $x^*$ and $x^{**}$ (where
LAS denotes locally asymptotically stable and a hyphen denotes that the equilibrium does not exist).

\begin{center}
\begin{tabular}{ | l | c | c | c | } \hline
 & $x_0$ & $x^*$ & $x^{**}$					\\ \hline
$\Rone, \Rtwo < 1$ & LAS & - & -				\\  
$\Rone < 1 < \Rtwo$ & unstable & - & LAS		\\
$\Rtwo < 1 < \Rone$ & unstable & LAS & -		\\
$1 < \Rtwo < \Rone$ & unstable & LAS & unstable	\\
$1 < \Rone < \Rtwo$ & unstable & unstable & LAS	\\ \hline
\end{tabular}
\end{center}
\end{theorem}

This theorem characterizes how the system behaves near each of the equilibria based on $\Rone$ and $\Rtwo$.
For instance, if $\Rone < 1$ and $ \Rtwo < 1$ and the initial number of adopters of each ideology is sufficiently
small, then both ideologies will disappear.  When both ideologies are able to survive in isolation
(i.e. $1 < \Rone, \Rtwo$), then the invasion numbers $\Ione$ and $\Itwo$ determine whether or
not a small amount of one ideology can successfully invade when the other is already established
at its equilibrium.

\subsection{The global dynamics}\label{subs:globalpicture}
The stability results given in the previous subsection were local in nature.  We now give the global
results.

\begin{theorem}
 If $\Rone <1$ and $ \Rtwo<1 $, then $x_0$ is globally asymptotically stable.
\end{theorem} 

\begin{proof} 
Note that $\left. S' \right|_{S=0} = \Lambda > 0$, and so we may assume that $S>0$.  We now study
the stability of $x_0$ using the Lyapunov function 
\[
 U = S_0 g \left( \frac{S}{S_0} \right) + \frac{c_{E_1}}{\ccc_1} E_1 + \frac{\mu + d_{E_1} + c_{E_1}}{\ccc_1} R_1 + 
 						\frac{c_{E_2}}{\ccc_2} E_2 + \frac{\mu + d_{E_2} + c_{E_2}}{\ccc_2} R_2
\]
where the function $g$ is given in \eqref{g}.
A computation similar to the one in the proof of Theorem \ref{thm:stable0} shows that 
\[
 U' = - \mu \frac{\left( S - S_0\right)^2}{S}
 		+ \frac{D_1}{\ccc_1} \left[ \Rone - 1 \right] R_1
 		+ \frac{D_2}{\ccc_2} \left[ \Rtwo - 1 \right] R_2.
\]
By LaSalle's principle it follows that if $\Rone < 1 $ and $ \Rtwo < 1 $ then 
$x_0$ is globally asymptotically stable.
\end{proof}

\begin{theorem}		\label{3 GAS}
If $\Rone > \max \left\{ 1, \Rtwo \right\}$, then $x^*$ is globally asymptotically stable.
If $\Rtwo > \max \left\{ 1, \Rone \right\}$, then $x^{**}$ is globally asymptotically stable.
\end{theorem} 

\begin{proof} 
Suppose that $\Rtwo > \max \left\{ 1, \Rone \right\}$.  (We prove this case because it will help in
Section~\ref{Cross-interaction}.)  Since $\Rtwo>1$, the equilibrium $x^{**}$ is non-negative.  Let
\[
 W= W_1 + W_2,
\]
with
\[
W_1 = \frac{c_{E_1}}{\ccc_1} E_1 + \frac{\mu + d_{E_1}+ c_{E_1}}{\ccc_1} R_1
\]
and
\[
W_2 = S^{**} g \left( \frac{S}{S^{**}}\right)
			+ A E_2^{**} g \left( \frac{E_2}{E_2^{**}} \right)
			+ B R_2^{**} g\left( \frac{R_2}{R_2^{**}} \right),
\]
where $A$ and $B$ are chosen so that $A q_{E_2} + B q_{R_2} = 1$ and
$B c_{E_2} E_2^{**} = A c_{R_2} R_2^{**} + A q_{E_2} \beta_2 S^{**} R_2^{**}$.
Recalling that $\ccc_1 = c_{E_1} + q_{R_1} \left( \mu + d_{E_1} \right)$, we find that
\begin{equation*}
W_1' = \beta_1 S R_1 - \frac{D_1}{\ccc_1} R_1.
\end{equation*}

\noindent
In the calculations that follow, we use the scaled variables $\left( s, e_2, r_2 \right)
= \left( \frac{S \;}{S^{**}}, \frac{E_2 \;}{E_2^{**}}, \frac{R_2 \;}{R_2^{**}}\right)$.
A computation similar to the one in the proof of Theorem \ref{thm:xast_stability} shows that
\begin{equation*}
\begin{aligned}
W_2' = & \mu S^{**} \left[ 2 - s - \frac{1}{s} \right] + A c_{R_2} R_2^{**} \left[ 2 - \frac{r_2} {e_2} - \frac{e_2} {r_2} \right]	\\
 	& + \beta_2 S^{**} R_2^{**} \left[ A q_{E_2} \left( 3 - \frac{1} {s} - \frac{e_2}{r_2} - \frac{s r_2} {e_2} \right)
					+ B q_{R_2} \left( 2 - s - \frac{1} {s} \right) \right]							\\
	& + \beta_1 R_1 \left( S^{**} - S \right)
\end{aligned} 
\end{equation*}
Thus, we have
\begin{equation*}
\begin{aligned}
W' = \mu S^{**} & \left[ 2 - s - \frac{1}{s} \right] + A c_{R_2} R_2^{**} \left[ 2 - \frac{r_2} {e_2} - \frac{e_2} {r_2} \right]	\\
 	& + \beta_2 S^{**} R_2^{**} \left[ A q_{E_2} \left( 3 - \frac{1} {s} - \frac{e_2}{r_2} - \frac{s r_2} {e_2} \right)
					+ B q_{R_2} \left( 2 - s - \frac{1} {s} \right) \right]							\\
	& + \frac{D_1}{\ccc_1} \left( \frac{\beta_1 \ccc_1}{D_1} S^{**} - 1 \right) R_1
\end{aligned} 
\end{equation*}

Recalling that $\Gamma_1 = \frac{\beta_1 \ccc_1}{D_1}$, $S^{**} = \frac{1}{\Gamma_2}$ and
$\mathcal{R}_i = \frac{\Lambda}{\mu} \Gamma_i$, it follows that
$\frac{\beta_1 \ccc_1}{D_1} S^{**} = \frac{\Gamma_1}{\Gamma_2} = \frac{\Rone}{\Rtwo} < 1$.
Using that fact, along with the arithmetic mean--geometric mean inequality, we have
\begin{equation}				\label{3W'}
\begin{aligned}
W' &= \mu S^{**} \left[ 2 - s - \frac{1}{s} \right] + A c_{R_2} R_2^{**} \left[ 2 - \frac{e_2} {r_2} - \frac{r_2} {e_2} \right]	\\
 	&\qquad + \beta_2 S^{**} R_2^{**} \left[ A q_{E_2} \left( 3 - \frac{1} {s} - \frac{e_2}{r_2} - \frac{s r_2} {e_2} \right)
					+ B q_{R_2} \left( 2 - s - \frac{1} {s} \right) \right]								\\
	&\qquad + \frac{D_1}{\ccc_1} \left( \frac{\Rone}{\Rtwo} - 1 \right) R_1								\\
&\le 0,
\end{aligned} 
\end{equation}
with equality only if $s=1$ (meaning $S = S^{**}$) and $R_1 = 0$.  The largest invariant set for which $S$ remains
constant at $S^{**}$ and $R_1$ remains constant at $0$ consists of just the equilibrium $x^{**}$.  Thus, by
LaSalle's Invariance Principle, we may conclude that $x^{**}$ is globally asymptotically stable whenever
$\Rtwo > \max \left\{ 1, \Rone \right\}$.

The proof that $x^*$ is globally asymptotically stable whenever $\Rone > \max \left\{ 1, \Rtwo \right\}$ is similar.
\end{proof}

Coexistence is not possible except for the degenerate case $\Rone = \Rtwo$.

\section{Multiple Ideology Model with Cross-interaction}		\label{Cross-interaction}
It is known that in nature many species of microorganisms stably coexist while interacting with each other.
In mathematical epidemiology several mechanisms for coexistence of competing strains of a pathogen have
been studied.  Of particular relevance for us is the mechanism of superinfection, that is, the process by which
an individual that has previously been infected by one strain of a pathogen becomes infected with a different
strain.  The second strain is assumed to ``take over" the infected individual immediately
\cite{martcheva2015introduction}.  Thus this individual is considered to be infected by only
the second strain of the pathogen.
In our considerations, the role of pathogens is replaced by ideologies.

\subsection{Equations}
We introduce a cross-interaction term to take into account that some of the extremists subscribing
to one ideology may switch to supporting the other ideology - but only in one direction.  We assume that
recruiters do not change from one ideologies to the other.

\begin{equation}\label{eqn:model2}
 \begin{aligned}
 S' & = \Lambda - \mu S - \beta_1 SR_1	- \beta_2 SR_2	\\
 E_1' & = q_{E_1}\beta_1 SR_1 - (\mu + d_{E_1} + c_{E_1}) E_1 + c_{R_1} R_1 	-\delta E_1 E_2 	\\
 R_1' & = q_{R_1} \beta_1 S R_1 + c_{E_1} E_1 - (\mu + d_{R_1} + c_{R_1}) R_1\\ 
 E_2 ' & = q_{E_2}\beta_2 SR_2 - (\mu + d_{E_2} + c_{E_2}) E_2 + c_{R_2} R_2	+\delta E_1 E_2 \\
 R_2' & = q_{R_2} \beta_2 S R_2 + c_{E_2} E_2 - (\mu + d_{R_2} + c_{R_2} )R_2 
 \end{aligned}
\end{equation}
where $q_{E_1} + q_{R_1} =  q_{E_2} + q_{R_2} = 1$, with $q_{E_1}, q_{R_1}, q_{E_2}, q_{R_2} \in [0,1]$.
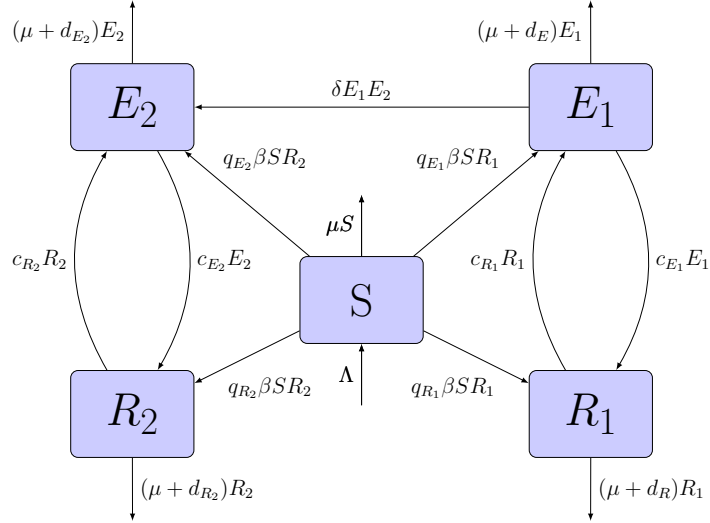
\begin{figure}[t] 
 \scalebox{0.7}{
\begin{tikzpicture}[node distance = 2cm,auto]
 \node [block] (S) {\Huge{S}}; 
 \node [block, below right=0.5cm and 2cm of S] (R1) {\Huge{$ R_1 $}};
 \node [block, above right=2cm and 2cm of S] (E1) {\Huge{$ E_1 $}};
 \node [block, below left=0.5cm and 2cm of S] (R2) {\Huge{$R_2$}};
 \node [block, above left=2cm and 2cm of S] (E2) {\Huge{$E_2$}};

 
 \path [line] (S) --node [below left, near end ]{$ q_{R_1} \beta SR_1 $} (R1);
 \path [line] (S) -- node [above left,near end ]{$q_{E_1} \beta S R_1 $} (E1); 
 \path [line] (S) -- node{$ \mu S $} (0,2);
 
 \path [line] (R1) edge [bend left] node [left]{$c_{R_1} R_1 $} (E1);
 \path [line] (E1) edge [bend left] node [right]{$c_{E_1} E_1 $} (R1);
 \path [line] (0,-2) -- node{$ \Lambda $} (S);
 
 \path [line] (R1) -- node [] {$(\mu + d_R) R_1 $} (4.35,-4.2);
 \path [line] (E1) -- node [] {$(\mu + d_E) E_1 $} (4.35,5.7);


 \path [line] (S) --node [below right, near end ]{$ q_{R_2} \beta SR_2 $} (R2);
 \path [line] (S) -- node [above right,near end ]{$q_{E_2} \beta S R_2 $} (E2); 
 \path [line] (S) -- node{$ \mu S $} (0,2);
 
 \path [line] (R2) edge [bend left] node [left]{$c_{R_2} R_2 $} (E2);
 \path [line] (E2) edge [bend left] node [right]{$c_{E_2} E_2 $} (R2);
 \path [line] (0,-2) -- node{$ \Lambda $} (S);
 
 \path [line] (R2) -- node [] {$(\mu + d_{R_2} ) R_2 $} (-4.35,-4.2);
 \path [line] (E2) -- node [] {$(\mu + d_{E_2} ) E_2 $} (-4.35,5.7);
 \path [line] (E1) -- node [above ]{$ \delta E_1 E_2 $} (E2);

\end{tikzpicture}
}
\caption{Transfer diagram for the model with two ideologies and cross-interaction.  \label{fig:transfer_diagram2}}
\end{figure}

Let $ T = S + R_1 + E_1 + R_2 + E_2$ be the total population, then we have the following proposition.

\begin{proposition}\label{prop:invariant_region2} 
 The region $\widetilde \Delta = \left \{ (S, E_1, R_1, E_2, R_2 ) \in \mathbb{R}^5_{\geq 0} : T \leq \frac{\Lambda} {\mu} \right \}$
 is positively invariant and attracting within $\mathbb{R}_{\ge 0}^5$ for the flow of \eqref{eqn:model2}.
\end{proposition}

The proof is similar to the proof of Proposition \ref{prop:invariant_region}.

\subsection{Equilibrium points}
This new system shares three equilibrium points with the previous system.  In fact, it is easy to see that the
system has the following equilibrium points that were defined in Section \ref{sec-3-equil}:
\begin{align*} 
 x_0 & = \left( \frac{\Lambda} {\mu}, 0, 0 ,0, 0 \right) \\
 x^{\ast} & = \left(S^*, E_1^*, R_1^*, 0, 0 \right) \\
 x^{**} & = \left( S^{**}, 0, 0, E_2^{**}, R_2^{**} \right) 
\end{align*} 
In some cases there is also a coexistence equilibrium, which will be studied later.

\subsection{Basic reproduction numbers and invasion numbers}		\label{brnain}
The basic reproduction numbers for this model coincide with the ones obtained for the previous model.
That is, we have 
\[
 \mathcal{R}_i = \frac{\Lambda}{\mu} \Gamma_i
 	= \beta_i \frac{\Lambda}{\mu}
		\frac{c_{E_i} + q_{R_i} \left( \mu + d_{E_i} \right)}{(\mu + d_{R_i} + c_{R_i}) (\mu + d_{E_i} + c_{E_i}) - c_{E_i} c_{R_i}}
\]
for $i = 1,2$.

\begin{proposition}
 The equilibrium $x_0$ is locally asymptotically stable if $\max \left\{ \Rone, \Rtwo \right\} < 1$,
 and is unstable if $\max \left\{ \Rone, \Rtwo \right\} > 1$.
\end{proposition} 

\noindent
{\bf The invasion number $\Itwod$:}
To compute the invasion number $\Itwod$ of the second ideology (near $x^*$, where the first ideology is
established), we note that the variables associated with the second ideology are $E_2$ and $R_2$.  We
have
\[
 \mathcal{F} = \begin{bmatrix}
 \mathcal{F}_{E_2} \\
 \mathcal{F}_{R_2} 
 \end{bmatrix} = 
 \begin{bmatrix}
 q_{E_2} \beta_2 SR_2 	+\delta E_1 E_2 \\
 q_{R_2} \beta_2 SR_2 
 \end{bmatrix} 
\;\;\; \textrm{and} \;\;\;
\mathcal{V} = \begin{bmatrix}
 \mathcal{V}_{E_2} \\
 \mathcal{V}_{R_2}
 \end{bmatrix} = 
 \begin{bmatrix}
 (\mu + d_{E_2} + c_{E_2}) E_2 - c_{R_2} R_2 \\
  - c_{E_2} E_2 + (\mu + d_{R_2}+ c_{R_2} ) R_2
 \end{bmatrix}.
\]
Differentiating and evaluating at $x^*$, we obtain
\[
 F = \begin{bmatrix}
 \delta E_1^* & q_{E_2} \beta_2 S^* \\
 0 & q_{R_2} \beta_2 S^*
 \end{bmatrix}
\aand
V = \begin{bmatrix}
 \mu +d_{E_2} + c_{E_2} & - c_{R_2} \\
 - c_{E_2} & \mu + d_{R_2} + c_{R_2}
 \end{bmatrix},
\]
and so the next generation matrix is
\begin{equation*}
\begin{aligned}
 N &= F V^{-1} 		\\
 &= \frac{1}{D_2} 
 \begin{bmatrix}
 \delta E_1^* & q_{E_2} \beta_2 S^* \\
 0 & q_{R_2} \beta_2 S^* 
 \end{bmatrix}
 \begin{bmatrix}
 \mu +d_{R_2} + c_{R_2} & c_{R_2} \\
 c_{E_2} & \mu + d_{E_2} + c_{E_2}
 \end{bmatrix}		\\
 &= \frac{1} { D_2} \begin{bmatrix}
 \delta E_1^* (\mu + d_{R_2} + c_{R_2}) + q_{E_2} \beta_2 S^* c_{E_2}
 					& \delta E_1^* c_{R_2} + q_{E_2} \beta_2 S^* (\mu + d_{E_2} + c_{E_2}) \\
 q_{R_2} \beta_2 S^* c_{E_2} & q_{R_2} \beta_2 S^* (\mu + d_{E_2} + c_{E_2}) \end{bmatrix}	\\
 &= \begin{bmatrix}
 A & B \\
 C & D
 \end{bmatrix} ,
\end{aligned}
\end{equation*}
where $A$, $B$, $C$, and $D$ are defined by this equation.  The invasion number of the second ideology
is the spectral radius of $N$, namely, 
\begin{equation*}
\begin{aligned}
 \Itwod &= \frac{(A + D )+ \sqrt{(A + D)^2 - 4(AD-BC)}}{2}	\\
 	&= \frac{(A + D )+ \sqrt{(A - D)^2 + 4BC}}{2}.
\end{aligned}
\end{equation*}
The case where $\delta=0$ is the same as the case studied in Section \ref{Section 3}, and it can be verified that
the expression given here gives $\left. \Itwod \right|_{\delta = 0} = \frac{\Rtwo}{\Rone} $, which agrees with the
expression for $\Itwo$ given in Section \ref{Section 3}.

We can view $A$ and $B$ as strictly increasing functions of $\delta$, while $C$ and $D$ do not depend on
$\delta$.  Since $A,B,C,D$ are all strictly positive, it is easy to verify that $\frac{\partial \Itwod}{\partial A}$ and
$\frac{\partial \Itwod}{\partial B}$ are positive.  Thus, $\Itwod$ is a strictly increasing function of $\delta$, and
so
\[
 \Itwod > \frac{\Rtwo} {\Rone}
\]
for $ \delta > 0 $.  This makes sense as it suggests that larger values of $\delta$ make it ``easier" for the
second ideology to invade a population where the first ideology is already established.
%

\vspace{0.4cm}
\noindent
{\bf The invasion number $\Ioned$:}
To compute the invasion number $\Ioned$ of the first ideology (near $x^{**}$, where the second ideology is
established), we note that the variables associated with the first ideology are $E_1$ and $R_1$.  In this case,
\[
 \mathcal{F} = \begin{bmatrix}
 \mathcal{F}_{E_1} \\
 \mathcal{F}_{R_1} 
 \end{bmatrix} = 
 \begin{bmatrix}
 q_{E_1} \beta_1 SR_1 \\
 q_{R_1} \beta_1 SR_1 
 \end{bmatrix} 
\;\;\; \textrm{and} \;\;\;
 \mathcal{V} = \begin{bmatrix}
 \mathcal{V}_{E_1} \\
 \mathcal{V}_{R_1}
 \end{bmatrix} = 
 \begin{bmatrix}
 (\mu + d_{E_1} + c_{E_1}) E_1 - c_{R_1} R_1 + \delta E_1 E_2 \\
 (\mu + d_{R_1}+ c_{R_1} ) R_1 - c_{E_1} E_1 
 \end{bmatrix}.
\]
Consequently,
\[
 F = \begin{bmatrix}
 0& q_{E_1}S^{**} \beta_1 \\
 0 & q_{R_1} S^{**} \beta_1 
 \end{bmatrix}
\aand
V = \begin{bmatrix}
 \mu +d_{E_1} + c_{E_1} + \delta E_2^{**} & - c_{R_1} \\
 - c_{E_1} & \mu + d_{R_1} + c_{R_1}
 \end{bmatrix},
\]
and the next generation matrix is
\[
 N =\frac{1}{ D_1 + \delta E_2^{**} (\mu + d_{R_1} + c_{R_1})} 
 \begin{bmatrix}
 0& q_{E_1}S^{**} \beta_1 \\
 0 & q_{R_1} S^{**} \beta_1 
 \end{bmatrix}
 \begin{bmatrix}
 \mu +d_{R_1} + c_{R_1} & c_{R_1} \\
 c_{E_1} & \mu + d_{E_1} + c_{E_1} + \delta E_2^{**}
 \end{bmatrix} .
\]
Since $F$ has rank one, the same is true of $N$.  Thus, $N$ has only one non-zero eigenvalue, which is positive
and which must be equal to the trace.  So, the trace of $N$ is equal to its spectral radius, and therefore is equal to
$\Ioned$.  That is,
\begin{align*} 
 \Ioned
    & =\frac{\beta_1 S^{**} [c_{E_1} + q_{R_1} (\mu + d_{E_1}) + q_{R_1} \delta E_2^{**} ]}
 	{ D_1 + \delta E_2^{**} (\mu + d_{R_1} + c_{R_1})}	\\
    & = \frac{\beta_1 [\ccc_1 + q_{R_1} \delta E_2^{**} ]}
    	{\Gamma_2 [ D_1 + \delta E_2^{**} (\mu + d_{R_1} + c_{R_1})]}.
\end{align*}
Differentiating with respect to $\delta$ yields 
\begin{align*} 
 \frac{ d\Ioned} { d \delta}
 	&= \frac{\beta_1 E_2^{**} \left[ q_{R_1} D_1 - \ccc_1 (\mu + d_{R_1} + c_{R_1}) \right]}
			 	{\Gamma_2 [ D_1 + \delta E_2^{**} (\mu + d_{R_1} + c_{R_1})]^2}		\\
 	&= - \frac{\beta_1 E_2^{**} c_{E_1} [c_{R_1} + q_{E_1} (\mu + d_{R_1}) ]}
			 	{\Gamma_2 [ D_1 + \delta E_2^{**} (\mu + d_{R_1} + c_{R_1})]^2}		\\
 	&< 0.
\end{align*}
Thus, $\Ioned$ is a strictly decreasing function of $\delta$.  Recalling from Section \ref{Section 3} that
$\left. \Ioned \right|_{\delta=0} = \frac{\Rone} {\Rtwo}$, we have the following inequality
\[
 \Ioned < \frac{\Rone} {\Rtwo},
\]
for $\delta > 0$.   This makes sense as it suggests that larger values of $\delta$ make it ``harder" for the
first ideology to invade a population where the second ideology is already established.

The following theorem follows from \cite[Theorem 2]{van2002reproduction}, which connects the local stability
of a boundary equilibrium to the basic reproduction number (or invasion number) calculated at that equilibrium.

\begin{theorem}		\label{local stuff}
 Suppose $\Rone >1$ and $\Rtwo > 1$.  Then
 \begin{enumerate}
 \item The equilibrium $x^*$ is locally asymptotically stable if and only if $\Itwod < 1$.
 \item The equilibrium $x^{**}$ is locally asymptotically stable if and only if $\Ioned < 1$.
 \item If $\Ioned$ or $\Itwod$ is greater than $1$, then the corresponding equilibrium is unstable.
 \end{enumerate} 
\end{theorem}
\begin{remark} 
We have shown that $ \Itwod > \frac{\Rtwo} {\Rone} $, and that $ \Ioned <\frac{\Rone} {\Rtwo} $.
As a consequence the region of the parameter space where the second ideology is locally asymptotically stable
is larger (in some sense) than that of the first ideology.  

A consequence of this theorem is that in order for the first ideology to eliminate the second (i.e. to out compete it),
a necessary condition is that $\Rone > \Rtwo$.  In contrast, the second ideology may be able eliminate the first one
if $\Rtwo > \Rone$, but also if $\Rtwo < \Rone$.  Thus, in the case of ``superinfection" (i.e. $\delta > 0$) the principle
that the ideology with the larger reproduction number excludes the one with the smaller reproduction number is no
longer valid.  Thus, an ideology (the second one) with suboptimal reproduction number may be able to dominate the
population.  In the presence of ``superinfection", individuals supporting ideology one may ``convert" to ideology two,
and this may give the competitive advantage to ideology two even if its reproduction number is smaller.  
\end{remark}

\subsection{A Discussion of Coexistence Equilibria}			\label{sec:coexistence}

In Section \ref{Section 3} it was shown that for $\delta = 0$, the only possible equilibria are $x_0$, $x^*$ and $x^{**}$,
and so there are no equilibria for which both ideologies are present.  For $\delta > 0$ it is possible to have an
equilibrium with both ideologies present.  We will refer to such an equilibrium as a {\it coexistence equilibrium}.
(Thus far, we have found no evidence of multiple coexistence equilibria.)

In this section, we will sometimes assume that $\Rone$, $\Rtwo$ or both are greater than one, so that $x^*$, $x^{**}$
or both exist and hence the invasion numbers $\Ioned$, $\Itwod$ or both are defined (and positive).  As was discussed
in the previous subsection, whether or not each of $\Ioned$ and $\Itwod$ is greater than $1$ affects the qualitative
behaviour of the system, in part by determining the local stabilities of $x^*$ and $x^{**}$.  We will now show that these
quantities also affect other properties of the system.

We will consider $\delta$ as a bifurcation parameter.  Let $\deltaone$ and $\deltatwo$, respectively, be the values of
$\delta$ that make $\Itwod$ and $\Ioned$ equal to $1$ (if such a value exists).  (Thus, $\deltaone$ arises from
a calculation involving $x^*$, while $\deltatwo$ arises from a calculation involving $x^{**}$.)  Then
\begin{equation*}
\deltaone =  \frac{1}{E_1^*}
		\frac{\ccc_2 D_2}
			{q_{R_2} D_2 \left( \Rone - \Rtwo \right) + c_{E_2} \left( c_{R_2} + q_{E_2} \left( \mu + d_{R_2} \right) \right)}
		\left( \Rone - \Rtwo \right)
\end{equation*}
and
\begin{equation*}
\deltatwo =  \frac{1}{E_2^{**}}
		\frac{\ccc_1 D_1}
			{q_{R_1} D_1 \left( \Rtwo - \Rone \right) + c_{E_1} \left( c_{R_1} + q_{E_1} \left( \mu + d_{R_1} \right) \right)}
		\left( \Rone - \Rtwo \right).
\end{equation*}

Suppose $\Rone > 1$.  Then $x^*$ exists and $\Itwod > 0$.  A sufficient (but not necessary) condition for
$\deltaone$ to be positive is $\Rtwo < \Rone$.  If $\delta = \deltaone > 0$, then $\Itwod$ is equal to $1$, and the Jacobian
at $x^*$ has $0$ as an eigenvalue.  If $\delta$ increases through $\deltaone$, then the stability of $x^*$ changes.  This
is the only value of $\delta$ that could lead to a coexistence equilibrium having a bifurcation interaction with $x^*$ (also
likely through a saddle node bifurcation).

Suppose $\Rtwo > 1$.  Then $x^{**}$ exists and $\Ioned > 0$.  Note that a necessary condition for $\deltatwo$ to be
positive is $\Rtwo < \Rone$.  If $\delta = \deltatwo > 0$, then $\Ioned$ is equal to $1$, and the Jacobian at $x^{**}$ has
$0$ as an eigenvalue.  If $\delta$ increases through $\deltatwo$, then the stability of $x^{**}$ changes from unstable
(for smaller values of $\delta$) to locally asymptotically stable (for larger values of $\delta$).  This is the only value of
$\delta$ that could lead to a coexistence equilibrium having a bifurcation interaction with $x^{**}$ (likely through a
saddle node bifurcation).

If $\deltaone$ and $\deltatwo$ are both positive (and distinct), then as $\delta$ increases from very small values to very
large values it will first pass through the smaller of $\deltaone$ and $\deltatwo$, and then through the larger.
This corresponds to a coexistence equilibrium entering the positive orthant by passing through one of $x^*$ and
$x^{**}$ and then leaving the positive orthant by passing through the other.

If $\deltaone, \deltatwo > 0$, then let
\begin{equation*}
\sigma = \deltatwo - \deltaone
\end{equation*}
(otherwise $\sigma$ is left undefined).  Depending
on the system parameters, it is possible to have $\sigma$ positive or negative (or zero on a set of measure zero).  The 
sign of $\sigma$ determines whether the coexistence equilibrium (which is absent for $\delta = 0$) appears in $\Rgez^5$
(as $\delta$ is increased) by passing through $x^*$ and leaves by passing through $x^{**}$, or the opposite.  That is, the
sign of $\sigma$ determines the large scale motion of the coexistence equilibrium.  Numerical evidence (not included here)
shows that it is possible to have $\sigma$ being either positive or negative.


We now consider different possible outcomes that may arise.  In this discussion it is useful to recall from
Section \ref{brnain} that $\frac{\partial \Ioned}{\partial \delta} < 0$ and $\frac{\partial \Itwod}{\partial \delta} > 0$.

\vspace{0.5cm} \noindent {\bf Situation 1:}
Suppose $1 < \Rone < \Rtwo$.  Under this assumption we have:
\begin{itemize}
	\item $x^*$ and $x^{**}$ both exist.
	\item For $\delta = 0$, we have:
	\begin{itemize}
		\item $\Itwod = \frac{\Rtwo}{\Rone} > 1$ and so $x^*$ is unstable.
		\item $\Ioned = \frac{\Rone}{\Rtwo} < 1$ and so $x^{**}$ is locally asymptotically stable.
		\item By Theorem \ref{3 GAS}, $x^{**}$ is globally asymptotically stable.
	\end{itemize}
\end{itemize}

Since $\Ioned$ and $\Itwod$ are monotone functions of $\delta$, it follows that $\Ioned < 1 < \Itwod$
for all $\delta > 0$.  Thus, $x^*$ is unstable and $x^{**}$ is locally asymptotically stable for all
$\delta > 0$, and there is no coexistence equilibrium.
It is plausible that $x^{**}$ is globally asymptotically stable for all $\delta > 0$.
(See Section \ref{subs: global 4}.)

%

\vspace{0.5cm} \noindent {\bf Situation 2:}
Suppose $1 < \Rtwo < \Rone$.  Under this assumption we have:
\begin{itemize}
	\item $x^*$ and $x^{**}$ both exist.
	\item For $\delta = 0$, we have:
	\begin{itemize}
		\item $\Ioned = \frac{\Rone}{\Rtwo} > 1$ and so $x^{**}$ is unstable.
		\item $\Itwod = \frac{\Rtwo}{\Rone} < 1$ and so $x^*$ is locally asymptotically stable.
		\item By Theorem \ref{3 GAS}, $x^*$ is globally asymptotically stable.
	\end{itemize}
	\item $\deltaone > 0$.
	\item $\deltatwo$ could be positive or negative.
\end{itemize}

In order to determine how the behaviour of the system changes as $\delta$ increases, it is necessary to consider
the size of $\delta$ relative to $\deltaone$ and $\deltatwo$.  Hence, we consider three cases based on how $\deltatwo$
compares to $0$ and $\deltaone$.

\vspace{0.5cm} \noindent {\bf Case 2A:}
Suppose $\deltatwo < 0$.  Then:
\begin{itemize}
\item $\Ioned > 1$ for all $\delta > 0$, and so $x^{**}$ is unstable for all $\delta > 0$, and the first ideology
		is always able to invade near $x^{**}$.
\item If $0 \le \delta < \deltaone$, then $x^*$ is locally asymptotically stable and the second ideology is unable
		to invade near $x^*$.  There is no coexistence equilibrium.
\item If $\deltaone < \delta$, then $x^*$ is unstable and the second ideology can invade near $x^*$.  With both
		$x^*$ and $x^{**}$ being unstable, there is a locally asymptotically stable coexistence equilibrium.
\end{itemize}
In this case, as $\delta$ increases from $0$ to very large values, a stable coexistence equilibrium appears by
passing through $x^*$.

\vspace{0.5cm} \noindent {\bf Case 2B:}
Suppose $0 < \deltatwo < \deltaone$, so that $\sigma < 0$.  Then:

\begin{itemize}
\item If $0 \le \delta < \deltatwo$, then $x^*$ is locally asymptotically stable and the second ideology is unable
		to invade near $x^*$, whereas $x^{**}$ is unstable and so the first ideology can invade near $x^{**}$.
		There is no coexistence equilibrium.
\item If $\deltatwo < \delta < \deltaone$, then $x^*$ and $x^{**}$ are both locally asymptotically stable and so
		neither ideology can invade near the equilibrium where the other ideology is dominant.  In this setting,
		there is an unstable coexistence equilibrium.
\item If $\deltaone < \delta$, then $x^*$ is unstable and $x^{**}$ is locally asymptotically unstable.
		There is no coexistence equilibrium.
\end{itemize}
In this case, as $\delta$ increases from $0$ to very large values, an unstable coexistence equilibrium appears by
passing through $x^{**}$, traverses through the positive orthant and disappears by passing through $x^*$.

\vspace{0.5cm} \noindent {\bf Case 2C:}
Suppose $\deltaone < \deltatwo$, so that $\sigma > 0$.  Then:

\begin{itemize}
\item If $0 \le \delta < \deltaone$, then $x^*$ is locally asymptotically stable and the second ideology is unable
		to invade near $x^*$, whereas $x^{**}$ is unstable and so the first ideology can invade near $x^{**}$.
		There is no coexistence equilibrium.
\item If $\deltaone < \delta < \deltatwo$, then $x^*$ and $x^{**}$ are both unstable and so each ideology can invade
		near the equilibrium where the other ideology is dominant.  In this setting, there is a stable coexistence
		equilibrium.
\item If $\deltatwo < \delta$, then $x^*$ is unstable and $x^{**}$ is locally asymptotically unstable.  There is no
		coexistence equilibrium.
\end{itemize}
In this case, as $\delta$ increases from $0$ to very large values, a stable coexistence equilibrium appears by
passing through $x^*$, traverses through the positive orthant and disappears by passing through $x^{**}$.

\vspace{0.2cm}
For small values of $\delta$, Cases 2B and 2C yield similar phase spaces in terms of the stability of $x^*$ and $x^{**}$,
as well as in the lack of a coexistence equilibrium.  This is also true for large values of $\delta$.  The difference arises
when $\delta$ lies between $\deltaone$ and $\deltatwo$, in which case the sign of $\sigma$ determines the local stability
of the coexistence equilibrium as well as it's general direction of motion - from $x^*$ to $x^{**}$ or the opposite.

\vspace{0.5cm} \noindent {\bf Situation 3:}
Suppose $\Rtwo < 1 < \Rone$.  Under this assumption we have:
\begin{itemize}
	\item $x^*$ exists - and so $\Itwod$ is defined and positive.
	\item $x^{**}$ does not exist - and so $\Ioned$ is not defined and neither is $\deltatwo$.
	\item For $\delta = 0$, we have:
	\begin{itemize}
		\item $\Itwod = \frac{\Rtwo}{\Rone} < 1$ and so $x^*$ is locally asymptotically stable.
		\item By Theorem \ref{3 GAS}, $x^*$ is globally asymptotically stable.
	\end{itemize}
	\item $\deltaone > 0$.
\end{itemize}

As $\delta$ is varied, we have:
\begin{itemize}
	\item If $0 \le \delta < \deltaone$, then $x^*$ is locally asymptotically stable.
	\item If $\deltaone < \delta$, then $x^*$ is unstable.  Then the second ideology is able to invade near $x^*$,
		even though it is unable to invade near $x_0$.
		(It would appear that the second ideology is more attractive to adherents of the first ideology than it is
		to individuals in susceptible group.)
		A locally asymptotically stable coexistence equilibrium exists.
\end{itemize}
In this situation, as $\delta$ increases, a stable coexistence equilibrium appears, and then remains no matter how large
$\delta$ becomes.

\vspace{0.5cm} \noindent {\bf Situation 4:}
Suppose $\Rone < 1 < \Rtwo$.  Under this assumption we have:
\begin{itemize}
	\item $x^{**}$ exists - and so $\Ioned$ is defined and positive.
	\item $x^*$ does not exist - and so $\Itwod$ is not defined and neither is $\deltaone$.
	\item For $\delta = 0$, we have:
	\begin{itemize}
		\item $\Ioned = \frac{\Rone}{\Rtwo} < 1$ and so $x^{**}$ is locally asymptotically stable.
		\item By Theorem \ref{3 GAS}, $x^{**}$ is globally asymptotically stable.
	\end{itemize}
	\item $\deltatwo < 0$.
\end{itemize}
In this situation, there is no qualitative change regarding the equilibria and their local behaviours as
$\delta$ varies from $0$ to $\infty$.

\begin{remark}
From numerical experiments, it appears that when the coexistence equilibrium exists, it is unique.
\end{remark}

%
%
%

\subsection{Global stability}			\label{subs: global 4}

\begin{theorem}
If $\Rone, \Rtwo \le 1$, then $x_0$ is globally asymptotically stable (for all $\delta \ge 0$).
\end{theorem}

\begin{proof}
By Proposition \ref{prop:invariant_region2}, the attractor for the system must be contained in the set 
$\widetilde \Delta$, and so it is sufficient to consider a Lyapunov function on this set.  In particular, we
may assume that the variables are non-negative and satisfy $S + E_1 + R_1 + E_2 + R_2 \le S_0$,
where $S_0 = \frac{\Lambda}{\mu}$.

Let
\begin{equation*}
U = E_1 + \frac{\mu + d_{E_1} + c_{E_1}}{c_{E_1}} R_1 + E_2 + \frac{\mu + d_{E_2} + c_{E_2}}{c_{E_2}} R_2.
\end{equation*}
Then a straight-forward calculation produces
\begin{equation*}
\begin{aligned}
U' & = E_1' + \frac{\mu + d_{E_1} + c_{E_1}}{c_{E_1}} R_1' + E_2' + \frac{\mu + d_{E_2} + c_{E_2}}{c_{E_2}} R_2'	\\
&= \frac{D_1}{c_{E_1}} \left( \frac{\ccc_1}{D_1} \beta_1 S - 1 \right) R_1
		+ \frac{D_2}{c_{E_2}} \left( \frac{\ccc_2}{D_2} \beta_2 S - 1 \right) R_2		\\
&= \frac{D_1}{c_{E_1}} \left( \Rone \frac{S}{S_0} - 1 \right) R_1 + \frac{D_2}{c_{E_2}} \left( \Rtwo \frac{S}{S_0} - 1 \right) R_2.
\end{aligned}
\end{equation*}
On the set $\widetilde \Delta$, this is less than or equal to zero, with equality only if $R_1 = R_2 = 0$.  The largest
invariant subset of $\widetilde \Delta$ that satisfies $R_1 = R_2 = 0$ is the singleton consisting of $x_0$.  Thus, by
LaSalles Invariance Principle, $x_0$ is globally asymptotically stable.
\end{proof}

In order for one of $x^*$ and $x^{**}$ to be globally asymptotically stable, it is necessary for that equilibrium
to be locally asymptotically stable and for the other to be unstable.  Thus, the following result follows from
Theorem \ref{local stuff} and from the comparisons between the invasion numbers and the ratios of the
reproduction numbers.

\begin{proposition}
Suppose $\delta > 0$.
\begin{enumerate}
\item In order for $x^*$ to be globally asymptotically stable, a necessary condition is that
	$\frac{\Rtwo} {\Rone} < \Itwod < 1 < \Ioned < \frac{\Rone} {\Rtwo}$.
\item In order for $x^{**}$ to be globally asymptotically stable, a necessary condition is that
	$\Ioned < \min \left\{ 1, \frac{\Rone}{\Rtwo} \right\}$ and $\max \left\{ 1, \frac{\Rtwo} {\Rone} \right\} < \Itwod$.
\end{enumerate}
\end{proposition}

Regarding the minimum and maximum that appear in the second statement of this proposition,
either both are equal to $1$ or neither is equal to $1$.

\begin{theorem}		\label{thm:GAS}
Suppose $\Rtwo > \max \left\{ 1, \Rone \right\}$ and $\frac{c_{E_2}}{\ccc_2} < \frac{c_{E_1}}{\ccc_1}$.
Then $x^{**}$ is globally asymptotically stable.
\end{theorem}

\begin{proof}
We begin by interpreting the Lyapunov calculation from the proof of Theorem \ref{3 GAS} as a calculation
for system (\ref{eqn:model2}) for the special case $\delta=0$.  We use the same Lyapunov function
\begin{equation*}
W = S^{**} g \left( \frac{S}{S^{**}}\right)
	+ \frac{c_{E_1}}{\ccc_1} E_1 + \frac{\mu + d_{E_1}+ c_{E_1}}{\ccc_1} R_1
			+ A E_2^{**} g \left( \frac{E_2}{E_2^{**}} \right)
			+ B R_2^{**} g\left( \frac{R_2}{R_2^{**}} \right),
\end{equation*}
where $A$ and $B$ are again chosen so that
\begin{equation*}
A q_{E_2} + B q_{R_2} = 1
\aand
B c_{E_2} E_2^{**} = A c_{R_2} R_2^{**} + A q_{E_2} \beta_2 S^{**} R_2^{**}.
\end{equation*}
Note that at $x^{**}$, the equation for $E_2'=0$ implies
$c_{R_2} R_2^{**} + q_{E_2} \beta_2 S^{**} R_2^{**} = \left( \mu + d_{E_2} + c_{E_2} \right) E_2^{**}$,
and so it follows that $A = \frac{c_{E_2}}{\ccc_2}$ and $B = \frac{\mu + d_{E_2} + c_{E_2}}{\ccc_2}$.

Let $\left( s, e_2, r_2 \right) = \left( \frac{S \;}{S^{**}}, \frac{E_2 \;}{E_2^{**}}, \frac{R_2 \;}{R_2^{**}}\right)$.
Then \eqref{3W'} gives an expression for $W'$ for the case where $\delta=0$.  By replicating the calculation
that produced \eqref{3W'}, for a general $\delta \ge 0$, we obtain
\begin{equation*}
\begin{aligned}
W' &= \mu S^{**} \left[ 2 - s - \frac{1}{s} \right] + A c_{R_2} R_2^{**} \left[ 2 - \frac{e_2} {r_2} - \frac{r_2} {e_2} \right]	\\
 	&\qquad + \beta_2 S^{**} R_2^{**} \left[ A q_{E_2} \left( 3 - \frac{1} {s} - \frac{e_2}{r_2} - \frac{s r_2} {e_2} \right)
					+ B q_{R_2} \left( 2 - s - \frac{1} {s} \right) \right]								\\
	&\qquad + \frac{D_1}{\ccc_1} \left( \frac{\Rone}{\Rtwo} - 1 \right) R_1								\\
	&\qquad + \delta E_1 E_2 \left[ - \frac{c_{E_1}}{\ccc_1} + A \left( 1 - \frac{E_2^{**}}{E_2 \;} \right) \right]
\end{aligned}
\end{equation*}
We now focus on the final term, which we call $\tilde \delta$:
\begin{equation*}
\begin{aligned} 
\tilde \delta &= \delta E_1 E_2 \left[ - \frac{c_{E_1}}{\ccc_1} + A \left( 1 - \frac{E_2^{**}}{E_2 \;} \right) \right]	\\
	&= \delta E_1 \left[ \left( A - \frac{c_{E_1}}{\ccc_1} \right) E_2 - A E_2^{**} \right].
\end{aligned}
\end{equation*}
Recalling that $A = \frac{c_{E_2}}{\ccc_2}$, it follows that the condition $\frac{c_{E_2}}{\ccc_2} < \frac{c_{E_1}}{\ccc_1}$
implies $\tilde \delta$ is less than or equal to zero.  Then, by combining the assumption that
$\Rtwo > \max \left\{ 1, \Rone \right\}$ with the Arithmetic Mean-Geometric Mean Inequality, it follows that
$W' \le 0$ on $\Rgtz^5$ with equality only if $s=1$ and $R_1=0$.  Using LaSalle's Invariance Principle, in
a similar manner to how it was used in the proofs of Theorem \ref{thm:xast_stability} and Theorem \ref{3 GAS},
it follows that $x^{**}$ is globally asymptotically stable in $\Rgtz^5$.
%
\end{proof}

\begin{remark}
In the preceding proof, it is sufficient to restrict our analysis to the set $\widetilde \Delta$, in which
$E_2 \le \frac{\Lambda}{\mu}$.  Thus, in $\widetilde \Delta$, $\tilde \delta \le 0$ (and therefore $W' \le 0$)
as long as $A \le \frac{\frac{\Lambda}{\mu}}{\frac{\Lambda}{\mu} - E_2^{**}} \frac{c_{E_1}}{\ccc_1}
= \frac{1}{1 - \frac{\mu}{\Lambda}E_2^{**}}\frac{c_{E_1}}{\ccc_1}$ (which is larger than $\frac{c_{E_1}}{\ccc_1}$).
This allows Theorem \ref{thm:GAS} to be restated with less restrictive (but messier) conditions.
\end{remark}

\begin{remark}
The portion of the parameter space for which Theorem \ref{thm:GAS} applies is a subset of the parameters
that result in Situations 1 and 4 in Section \ref{sec:coexistence}.  In Case 1B, we saw that if $1 < \Rone < \Rtwo$
and $0 < \deltaone < \delta$, then $x^{**}$ is not globally asymptotically stable, and so
Theorem \ref{thm:GAS} must not apply.  In that case, the conversion from ideology one to ideology two
(i.e. having $\delta > 0$) has fundamentally changed the global dynamics from how they were with
$\delta=0$.
\end{remark}

\section*{Acknowledgments}
We would like to thank Lorne Dawson for an interesting discussion on the topic of radicalization.  


\bibliography{rad_papers}{}
\bibliographystyle{siam}

\end{document}